\documentclass[11pt]{amsart}

\usepackage{amssymb,latexsym,amsthm, amsmath, comment, fullpage}
\usepackage[all]{xy}
\usepackage{graphicx, comment}

\ifx\pdfoutput\undefined
\usepackage{hyperref}
\else
\usepackage[pdftex,colorlinks=true,linkcolor=blue,urlcolor=blue]{hyperref}
\fi
\theoremstyle{plain}
\newtheorem{theorem}{Theorem}
\newtheorem{corollary}{Corollary}
\newtheorem{proposition}{Proposition}

\newtheorem{lemma}{Lemma}
\theoremstyle{definition}
\newtheorem{definition}{Definition}
\newtheorem{remark}{Remark}

\newtheorem*{ack}{Acknowledgements}

\begin{document}

\title[Traces of random operators associated with self-affine Delone sets]
      {Traces of random operators associated with self-affine Delone sets and Shubin's formula}
\author[S. Schmieding]{Scott Schmieding}
\address{Northwestern University}
\email{schmiedi@math.northwestern.edu}
\author[R. Trevi\~{n}o]{Rodrigo Trevi\~no}
\address{University of Maryland}
\email{rodrigo@math.umd.edu}
\begin{abstract}
We study operators defined on a Hilbert space defined by a self-affine Delone set $\Lambda$ and show that the usual trace of a restriction of the operator to finite-dimensional subspaces satisfies a certain $\limsup$ law controlled by traces on a certain subalgebra. The asymptotic traces are defined through asymptotic cycles, or $\mathbb{R}^d$-invariant distributions of a dynamical system defined by $\Lambda$. We use this to refine Shubin's trace formula for certain self-adjoint operators acting on $\ell^2(\Lambda)$ and show that the errors of convergence in Shubin's formula are given by these traces.
\end{abstract}
\maketitle
\section{Introduction and statement of results}
This paper is about traces on algebras of random operators associated to dynamical systems defined by aperiodic, self-affine Delone sets. Algebras of operators close to the ones studied here have been considered before by different authors (e.g. \cite{kellendonk:noncom, BHZ:Hulls, LS:algebras}) and they are motivated in part by the study of spectral properties of Schr\"odinger operators arising in the study of quasicrystals. Since quasicrystals are modeled by Delone sets $\Lambda$, which are uniformly discrete subsets of $\mathbb{R}^d$, this leads to the study of self-adjoint operators on Hilbert spaces of the form $\ell^2(\Lambda)$ defined by $\Lambda$.

More specifically, let $A$ be a self-adjoint operator on $\ell^2(\Lambda)$ and $A|_{B_T}$ be the restriction of $A$ to the subspace $\ell^2(\Lambda\cap B_T)$, where $B_T$ is a ball of radius $T$ around the origin. For any $E>0$, define
$$n_T^A(E) = \frac{\#\{\mbox{ eigenvalues of $A|_{B_T}$ less than or equal to }E\}}{\mathrm{Vol}(B_T)}.$$
The function $E\mapsto n_T^A(E)$ is the distribution function of the measure $\rho_T^A$ defined as
$$\int\varphi\rho_T^A = \frac{\mathrm{tr}\,(\varphi(A|_{B_T}))}{\mathrm{Vol}(B_T)}$$
for $\varphi\in C_0(\mathbb{R})$. It turns out that there exists a unique trace $\tau: C^*(\Lambda)\rightarrow\mathbb{C}$ such that $\rho_T^A(\varphi)\rightarrow \rho_A(\varphi)$ in the weak topology, where the measure $\rho_A$ is defined as $\rho_A(\varphi) = \tau(\varphi(A))$ and $C^*(\Lambda)$ is a C$^*$-algebra defined by $\Lambda$ (see e.g. \cite{bellissard:gaps, LS:algebras} and references therein). The limiting distribution of the measure $\rho_A$ is called the \textbf{integrated density of states}. Shubin's trace formula
\begin{equation}
\label{eqn:Shubin}
 \lim_{T\rightarrow \infty}\frac{1}{\mathrm{Vol}(B_T)}\mathrm{tr}\,(\varphi(A|_{B_T})) = \tau(\varphi(A)),
\end{equation}
which asserts that the trace per unit volume gives the integrated density of states, gives information about the structure of the spectrum $\sigma(A)$. Shubin's formula was first proved for pseudodifferential operators with almost periodic coefficients and it was extended by Bellissard to cases covering Delone sets \cite{bellissard:Ktheory}, which is the setting we consider here.

The goal of this paper is to refine the convergence in (\ref{eqn:Shubin}). For example, we want to know whether there is an asymptotic statement that can be made for the difference
$$\mathrm{tr}\,(\varphi(A|_{B_T})) - \rho_A(\varphi) \mathrm{Vol}\,(B_T)$$
as $T\rightarrow \infty$.
In this case, any statement would yield information on the error rates of the integrated density of states given by self-adjoint operators $A:\ell^2(\Lambda)\rightarrow \ell^2(\Lambda)$. By applying the results of \cite{ST:SA}, which gives rates of converge of ergodic averages of systems related to self-affine sets, in this paper, we will show that as long as the quasicrystal is self-affine, the convergence in Shubin's formula (\ref{eqn:Shubin}) can be refined since deep down it can be obtained through an ergodic average.

More precisely, the route to this refinement is the study of deviations of ergodic averages of uniquely ergodic $\mathbb{R}^d$-actions on a compact metric space $\Omega_\Lambda$ defined by a self-affine Delone set $\Lambda$, where the action consists of translating the set $\Lambda$. The types of Delone sets which are studied are called \textbf{renormalizable of finite type} (RFT Delone sets), which are defined in \S \ref{subsec:affinity}. Examples of such Delone sets are given by aperiodic substitution tilings such as the Penrose tilings and the Ammann-Beenker tilings in two dimensions, and the icosahedral tilings in three dimensions. There are also examples coming from the cut and project construction which are not immediately given by substitution tilings.

We briefly summarize the setup from \cite{ST:SA} relevant in order to state the main results of this paper (see \S \ref{sec:TracesCycles} for details). Given any RFT Delone set $\Lambda$, there exist $d_\Lambda\geq 1$ numbers $\nu_1>|\nu_2|\geq \cdots \geq |\nu_{d_\Lambda}| > 1$ which are the eigenvalues of an induced map on the cohomology space $H^d(\Omega_\Lambda;\mathbb{R})$ (this is defined in \S\ref{sec:cohomology}). Along with these numbers are cohomology classes $[\eta_{i,j,k}]$ which are generalized eigenvectors of the induced action on cohomology and are represented by $d$-forms $\eta_{i,j,k}$ on $\mathbb{R}^d$. These classes are indexed by a set $I_\Lambda^+$ which is finite by the RFT assumption. There are dual currents $\mathfrak{C}_{i,j,k}$ such that, roughly speaking, if $\mathfrak{C}_{i,j,k}(\eta)\neq 0 $ then $\left|\int_{B_T}\eta\right|$ is at least of the order of $T^{d s_i}$, where $s_i = \frac{\log |\nu_i|}{\log |\nu_1|}$. The set $I_\Lambda^+$ is partially ordered: $(i,j,k)\leq (i',j',k')$ if $L(i,j,T)T^{ds_i}\geq L(i',j',T)T^{ds_{i'}}$ for any (and all) $T>1$, where $L(i,j,T)$ are powers of $\log T$ defined in (\ref{eqn:currents1}) and (\ref{eqn:currents2}) (the order does not depend on the indices $k$).

While the leading current $\mathfrak{C}_{1,1,1}$ represents the unique $\mathbb{R}^d$-invariant measure for the translation action on $\Omega_\Lambda$, the other currents capture a dynamical invariant for smooth functions. From the invariant currents $\mathfrak{C}_{i,j,k}$ we obtain traces $\tau_{i,j,k}$ on a certain $*$-algebra $\mathcal{A}_\Lambda^{tlc}\subset C^*(\Lambda)$ of bounded operators on $\ell^2(\Lambda)$ called $\Lambda$\textbf{-equivariant operators of finite range} (see Proposition \ref{prop:traces}). This $*$-algebra of operators contain many operators of interest to physics, such as Hamiltonian operators of the form $H = -\triangle + V$, for $\Lambda$-equivariant potentials $V$. As such, the trace $\tau_{1,1,1}$ is the trace $\tau$ from Shubin's trace formula (\ref{eqn:Shubin}). In what follows, given a bounded set $B_0\subset \mathbb{R}^d$ with a regular boundary (a \emph{good Lipschitz domain} as defined in \cite[\S 5]{ST:SA}), $B_T$ is a one-parameter family of sets obtained from $B_0$ by multiplying by a one parameter family of matrices $g_T$ in such a way that $\mathrm{Vol}(B_T) = \mathrm{Vol}(B_0)T^d$. We now give our refinement to Shubin's formula.
\begin{theorem}
  \label{thm:shubin}
  Let $\Lambda$ be an RFT Delone set, $B_0$ a good Lipschitz domain and $A\in\mathcal{A}_\Lambda^{tlc}$ a self-adjoint operator. For any index $(i,j,k)$ there exists a regular countably additive Borel measure $\rho_{i,j,k}^A$ such that for any polynomial $\varphi\in C(\mathbb{R})$,
\begin{equation}
  \label{eqn:shubin2}
  \begin{split}
    \limsup_{T\rightarrow \infty}& \frac{1}{L(i,j,T)T^{d \frac{\log |\nu_i|}{\log \nu_1}}}\left(  \mathrm{tr}(\varphi(A|_{B_T}))  -   \sum_{\substack{(i',j',k')\leq (i,j,k) \\ k'\neq k  }} \rho_{i',j',k'}^A(\varphi) \Psi_{i',j',k'}^{B_0}(T)L(i',j',T)T^{d \frac{\log |\nu_{i'}|}{\log \nu_1}}  \right) \\
    &= \rho^A_{i,j,k}(\varphi),
    \end{split}
\end{equation}
where $L(i,j,T)$ is a non-negative integer power of $\log T$ and  $\Psi_{i,j,k}^{B_0}:\mathbb{R}^+\rightarrow \mathbb{R}$ is a continuous bounded function satisfying $\limsup_{T\rightarrow\infty} \Psi_{i,j,k}^{B_0}(T) = 1 $. The measures are defined by $\rho_{i,j,k}^A(\varphi) = \tau_{i,j,k}(\varphi(A))$, where $\tau_{i,j,k}$ are traces on $\mathcal{A}^{tlc}_\Lambda$.
\end{theorem}
The functions $\Psi_{i,j,k}^{B_0}(T)$ describe the oscilations which inevitably happen as one integrates functions over the sets $B_T$. This will become more clear in the proof of the theorem.

Theorem \ref{thm:shubin} gives rates of convergence to Shubin's formula (\ref{eqn:Shubin}) in the case $\varphi$ is a polynomial, that is, errors of the integrated density of states: the distributions of the measures $\rho_{i,j,k}^A$ capture the error in the convergence to the integrated density of states. We do not know what physical interpretation the measures $\rho_{i,j,k}^A$ or their distributions may have. We do not know, at the moment, how to extend the $\limsup$ statement (\ref{eqn:shubin2}) in Theorem \ref{thm:shubin} to all continuous functions $\varphi \in C(\mathbb{R})$ (see Remark \ref{remark:allcontinuous}), although the functionals $\rho_{i,j,k}^A$ are measures. A more general theorem on traces on the $*$-algebra of operators $\mathcal{A}_\Lambda^{tlc}$ can also be derived from the results of \cite{ST:SA}.
\begin{theorem}
  \label{thm:main}
  Let $\Lambda$ be an RFT Delone set and $B_0$ a good Lipschitz domain. For every index $(i,j,k)\in I^+_\Lambda$ there exists a trace $\tau_{i,j,k}:\mathcal{A}_\Lambda^{tlc}\rightarrow \mathbb{C}$ such that for any $A\in\mathcal{A}_\Lambda^{tlc}$
  \begin{equation}
    \label{eqn:main}
    \begin{split}
      \limsup_{T\rightarrow \infty}&\frac{1}{L(i,j,T)T^{d \frac{\log |\nu_i|}{\log |\nu_1|}}}\left(  \mathrm{tr}(A|_{B_T})  -  \sum_{\substack{(i',j',k')\leq (i,j,k) \\ k'\neq k  }} \tau_{i',j',k'}(A) \Psi_{i',j',k'}^{B_0}(T) L(i',j',T)T^{d \frac{\log |\nu_{i'}|}{\log \nu_1}} \right) \\
      &= \tau_{i,j,k}(A),
    \end{split}
  \end{equation}
  where $L(i,j,T)$ is a non-negative power of $\log T$ and $\Psi_{i,j,k}^{B_0}:\mathbb{R}^+\rightarrow \mathbb{R}$ is  a continuous bounded function satisfying $\limsup_{T\rightarrow\infty} \Psi_{i,j,k}^{B_0}(T) = 1$.
\end{theorem}
It is important to point out that the traces $\tau_{i,j,k}$, defined on a dense $*$-subalgebra $\mathcal{A}_\Lambda^{tlc}$ of the C$^*$-algebra $C^*(\Lambda)$, \textbf{do not extend to the full C$^*$-algebra}, with the exception of $\tau_{1,1,1}$, which comes from the unique $\mathbb{R}^d$-invariant measure for the translation action on $\Omega_\Lambda$. This is due to the fact that the currents $\mathfrak{C}_{i,j,k}$, from which the traces are defined, are functionals defined only on forms possessing sufficient regularity, not just continuous forms. Thus, the subalgebra $\mathcal{A}^{tlc}_\Lambda$ corresponds to the operators in $C^*(\Lambda)$ which are ``smooth'' in some sense \footnote{We thank I. Putnam for pointing this out to us.}. The results above relate traces of self-adjoint operators to more than one homology class (closed cycles) of $\mathrm{Hom}(H^d(\Omega_\Lambda;\mathbb{R});\mathbb{R})$ and they echo the spirit of Connes's cyclic homology.

In recent years there has been a string of results for one-dimensional self-similar tilings, amongst which the Fibonacci Hamiltonian, defined by the Fibonacci substitution, has been thoroughly studied (see \cite{DEG:survey} for a comprehensive survey of the results). Many of the results on the spectral properties of Schr\"odinger operators concern ones coming from Sturmian substitutions. It is known that in such case, $\mathrm{dim}\,H^1(\Omega_\Lambda;\mathbb{R}) = 2$ and that $|I_\Lambda^+| = 1$, so our theorems do not say anything about Shubin's formula in such contexts; there is no refinement in those cases. Since the cohomology of tilings and Delone sets capture some sort of complexity, the Sturmian substitutions are not complex enough to have non-trivial limits in (\ref{eqn:shubin2}) and (\ref{eqn:main}). However, it is easy to construct substitutions on more than 2 symbols for which there are expansions for our main theorems (see \S \ref{subsec:examples}). For higher dimensions, operators coming from well-known substitution tilings such as the Penrose or Ammann-Beenker tilings do admit a refinement since $|I_\Lambda^+| = 3$ in those cases.

This paper is organized as follows. In section \ref{sec:delone} we recall Delone sets, pattern spaces and the translation action on pattern spaces. In section \ref{sec:cohomology} we go over the necessary definitions of cohomology for pattern spaces as well as RFT Delone sets. In section \ref{sec:ops} we recall the algebras of operators with which we will work and relate them to the cohomology spaces defined in the previous section. In section \ref{sec:TracesCycles} we recall the relevant results of \cite{ST:SA}, show that we can define traces from $\mathbb{R}^d$-invariant distributions on the pattern space from \cite{ST:SA}, and prove the main theorems. In the last section we work out an explicit example in one dimension.
\begin{ack}
We would like to thank J. Kellendonk for discussing questions related to Shubin's formula and I. Putnam for discussing traces on $*$-algebras. Part of R.T.'s travel funding for this project came from a AMS-Simons Travel Grant.
\end{ack}
\section{Delone sets}
\label{sec:delone}
A subset $\Lambda\subset \mathbb{R}^d$ is a \textbf{Delone set} if it satisfies two conditions:
\begin{enumerate}
\item \textbf{Uniformly discrete:} There exists a $r>0$ such that any distinct two points $x,y\in\Lambda$ are separated by a distance of at least $r$;
\item \textbf{Relatively dense:} There exists an $R>0$ such that for any point $x\in\mathbb{R}^d$, the ball of radius $R$ centered at $x$ contains at least one other point of $\Lambda$.
\end{enumerate}
The radius $r$ involved in uniform discreteness is called the \textbf{packing radius}; the radius $R$ involved in relative density is called the \textbf{covering radius}. A \textbf{cluster} is a finite subset of $\Lambda$. A Delone set has \textbf{finite local complexity} if for any given $R>0$, the set of all clusters found in any ball of radius $R$, up to translation, is finite.

For any Delone set $\Lambda$, we denote by $\varphi_t(\Lambda)$ the translation of the set $\Lambda$ by the vector $t\in\mathbb{R}^d$. For any two translates $\Lambda,\Lambda'$ of $\Lambda$, we define the distance between them by
$$d(\Lambda,\Lambda') = \inf\{\varepsilon>0: B_{\varepsilon^{-1}}(0)\cap \varphi_x(\Lambda) = B_{\varepsilon^{-1}}(0)\cap \varphi_y(\Lambda') \mbox{ for some }x,y\in B_\varepsilon(0) \}.$$
The \textbf{pattern space of $\Lambda$} is the closure of the set of translates of $\Lambda$ with respect to the above metric:
$$\Omega_\Lambda = \overline{\{\varphi_t(\Lambda): t\in\mathbb{R}^d\}}.$$
The \textbf{canonical transversal} of the pattern space $\Omega_\Lambda$ is the set
$$\mho_\Lambda = \{\Lambda'\in\Omega_\Lambda : 0\in\Lambda'\}.$$
If $\Lambda$ has finite local complexity, the canonical transversal $\mho_\Lambda$ is a Cantor set (i.e. perfect and totally disconnected) and the pattern space $\Omega_\Lambda$ is compact. In that case, the topology induced by the metric on $\mho_\Lambda$ is generated by clopen sets given by specifying clusters of $\Lambda$. That is, for any given cluster $C\subset \Lambda$ and a point $p\in C$, the set $U_C\subset \mho_\Lambda$ given by all patterns $\Lambda'\in \Omega_\Lambda$ with a cluster equivalent to $C$ around the origin (and $p$ identified to the origin) is a clopen set in the topology.

Thus, when $\Lambda$ has finite local complexity, the pattern space $\Omega_\Lambda$ has a local structure modeled on sets of the form $B_\varepsilon \times \mathcal{C}$ where $B_\varepsilon\subset \mathbb{R}^d$ is an open ball and $\mathcal{C}$ is a Cantor set. The pattern space is a foliated space where the leaves of the foliation are orbits of the $\mathbb{R}^d$ action.

The only types of Delone sets which will be considered here are those which give rise to a uniquely ergodic $\mathbb{R}^d$-action on $\Omega_\Lambda$, that is, sets for which there is a unique $\mathbb{R}^d$-invariant probability measure $\mu$ on $\Omega_\Lambda$ which is invariant under the translation action $\varphi_t$ of $\mathbb{R}^d$. Moreover, for any $f\in C(\Omega_\Lambda)$ and sequence $\{B_T\}$ of balls of radius $T$ centered at the origin,
$$\frac{1}{\mathrm{Vol}(B_T)}\int_{B_T} f\circ \varphi_t(\Lambda_0)\, dt \longrightarrow \int_{\Omega_\Lambda} f(\Lambda')\, d\mu(\Lambda')$$
uniformly for any $\Lambda_0$\footnote{The convergence for unique ergodicity holds more generally for \emph{F\o lner sequences} $\{B_T\}$ of sets, which are good sequences of averaging sets. See \cite[\S 8.4]{EW:book}.}. By the local product structure of $\Omega_\Lambda$, the invariant measure has a local product structure of the form $\mu = \mathrm{Leb}\times \mathfrak{m}_\Lambda$, where $\mathfrak{m}_\Lambda$ is a measure on $\mho_\Lambda$. Given a cluster $C\subset \Lambda$, the measure $\mathfrak{m}_\Lambda(U_C)$ is the asymptotic frequency of the cluster $C$ in $\Lambda$ \cite{LMS02} (this measure will be seen in the examples of \S \ref{subsec:examples}). These systems are always minimal in the sense that every leaf of the foliation of $\Omega_\Lambda$ is dense in $\Omega_\Lambda$.
\section{Cohomology}
\label{sec:cohomology}
Let $\Lambda\subset \mathbb{R}^d$ be a Delone set.
\begin{definition}
\label{def:PE}
A continuous function $f:\mathbb{R}^d\rightarrow \mathbb{R}$ is \textbf{$\Lambda$-equivariant} \cite{Kellendonk:PEC} if there exists an $R_f>0$ such that
$$B_{R_f}(0)\cap \varphi_x(\Lambda) = B_{R_f}(0)\cap \varphi_y(\Lambda) \mbox{ implies }f(x)=f(y).$$
\end{definition}
Differential forms which are $\Lambda$-equivariant are defined as differential forms for which the coefficients are $\Lambda$-equivariant functions. We denote by $\Delta_\Lambda^k$ the set of all $C^\infty$ $k$-forms which are $\Lambda$-equivariant. The complex $0\rightarrow \Delta_\Lambda^0\rightarrow \Delta_\Lambda^1\rightarrow \cdots \rightarrow \Delta_\Lambda^d\rightarrow 0$ is a subcomplex of the de Rham complex.
\begin{definition}
  \label{def:coh}
The $k^{th}$ $\Lambda$-equivariant cohomology spaces are defined by
$$H^k(\Omega_\Lambda;\mathbb{R}) = \frac{\mathrm{ker}\, \{d:\Delta_\Lambda^k\rightarrow \Delta_\Lambda^{k+1}\}}{\mathrm{Im}\, \{d:\Delta_\Lambda^{k-1}\rightarrow \Delta_\Lambda^{k}\}}.$$

The set $C^{\infty}_{tlc}(\Omega_\Lambda)$ is the set of \textbf{transversally locally constant functions}, that is, the set of continuous functions on $\Omega_\Lambda$ which are ($C^\infty$) smooth in the leaf direction of the foliation and locally constant in the transverse direction. For any such function, for any $\Lambda'$ there exists an $R_{\Lambda'}>0$ such that if for any $\Lambda''$ with $\Lambda'\cap B_{R_{\Lambda'}}(0) = \Lambda''\cap B_{R_{\Lambda'}}(0)$, then $f(\Lambda') = f(\Lambda'')$. For any $\Lambda'\in\Omega_{\Lambda}$, by \cite{Kellendonk-Putnam:RS} there is an algebra isomorphism $i_{\Lambda'}:C^{\infty}_{tlc}(\Omega_\Lambda) \rightarrow \Delta_\Lambda^0$ given, for any $h \in C^\infty_{tlc}(\Omega_\Lambda)$, by
\begin{equation}
\label{eqn:AlgIso}
f_h (t) := i_{\Lambda'}(h)(t) = h\circ \varphi_t(\Lambda').
\end{equation}
By the local product structure, any $h\in C^\infty_{tlc}(\Omega_\Lambda)$ defines a locally constant function on the canonical transversal $g_h:\mho_\Lambda\rightarrow \mathbb{R}$. We denote the set of locally constant functions on $\mho_\Lambda$ by $C^\infty(\mho_\Lambda)$.

To any $\Lambda$-equivariant function (and by the isomorphism above, to any function in $C^\infty_{tlc}(\Omega_\Lambda)$) we can assign a unique cohomology class in $H^d(\Omega_\Lambda;\mathbb{R})$. This is done as follows: let $f\in\Delta_\Lambda^0$ and denote by $(\star 1) = dx_1\wedge\cdots \wedge dx_d$ the smooth volume form in $\mathbb{R}^d$. Then $f(\star 1)$  is a closed form in $\Delta_\Lambda^d$ and therefore it has a cohomology class $[f(\star 1)]$. The \textbf{cohomology class of $f$} is defined to be $[f(\star 1)]\in H^d(\Omega_\Lambda;\mathbb{R}^d)$.
\end{definition}

Given a locally constant function $g\in C^\infty(\mho_\Lambda)$ we can extend it to a transversally locally constant function $h\in C_{tlc}^{\infty}(\Omega_\Lambda)$ by ``smoothing'' the function along the foliation direction in $\Omega_\Lambda$. By the isomorphism (\ref{eqn:AlgIso}), this extension defines a $\Lambda$-equivariant function $f_g\in\Delta_\Lambda^0$. As such, we can define the cohomology class of $g\in C^\infty(\mho_\Lambda)$ to be the cohomology class $[f_g]$ of $f_g$. The resulting cohomology class is independent of the bump function used to smooth out $g$ as long as it is supported on a small enough compact set and has integral one. Thus, characteristic functions $\chi_{U}\in C^\infty(\mho_\Lambda)$ of clopen sets $U\subset \mho_\Lambda$ have cohomology classes and linear functionals on $C^\infty(\mho_\Lambda)$ can be identified with linear functionals on $H^d(\Omega_\Lambda;\mathbb{R})$, that is, with homology classes.
\subsection{Self-affinity}
\label{subsec:affinity}
We recall some definitions from \cite{ST:SA}. Let $\Lambda\subset \mathbb{R}^d$ be a Delone set for which the $\mathbb{R}^d$ action on $\Omega_\Lambda$ is uniquely ergodic and denote by $\mu$ the unique $\mathbb{R}^d$-invariant measure. In that case, $\Lambda$ is \textbf{renormalizable of finite type} (RFT) if
\begin{enumerate}
\item There exists an expanding matrix $M_\Lambda \in GL^+(d,\mathbb{R})$ and a $\mu$-preserving homeomorphism $\Phi_\Lambda:\Omega_\Lambda\rightarrow \Omega_\Lambda$ satisfying the conjugacy
\begin{equation}
\label{eqn:conjugacy}
\Phi_\Lambda \circ \varphi_t = \varphi_{M_\Lambda t}\circ \Phi_\Lambda
\end{equation}
for any $t\in\mathbb{R}^d$. By expanding, we mean that $M_\Lambda$ has all eigenvalues, in modulus, greater than one.
\item The spaces $H^*(\Omega_\Lambda;\mathbb{R})$ are finite dimensional.
\end{enumerate}
\begin{remark}
There are plenty RFT Delone sets. For example, substitution tilings, by \cite{AP}, have finite dimensional cohomology spaces. Moreover, the substitution action induces the $\mu$-preserving homeomorphism of $\Omega_\Lambda$. Cut and project Delone sets can also be constructed to be RFT \cite[\S 8]{ST:SA}.
\end{remark}
For a RFT Delone set $\Lambda$ there is an induced action $\Phi_\Lambda^*: H^d(\Omega_\Lambda;\mathbb{R}) \rightarrow H^d(\Omega_\Lambda;\mathbb{R})$. Since $H^*(\Omega_\Lambda;\mathbb{R})$ is finite dimensional, we list the finite set of eigenvalues in decreasing order by norm: $\nu_1 > |\nu_2|\geq \cdots \geq |\nu_{d_\Lambda}|$, where $d_\Lambda = \mathrm{dim} H^d(\Omega_\Lambda;\mathbb{R})$. The spectral gap $\nu_1 > |\nu_2|$ is a consequence of the self-affine property and, moreover, we have that $\nu_1 = \mathrm{det}(M_\Lambda)$ \cite[\S 5.3.1]{ST:SA}. We also list the eigenvalues of $M_\Lambda$ by norm: $|\lambda_1|\geq \cdots \geq |\lambda_d| >1$.

By \cite[Lemma 1]{ST:SA}, the induced action by $\Phi^*_\Lambda$ is defined as follows. Let $\Lambda$ be a RFT Delone set and $[\eta]$ a class in $H^d(\Omega_\Lambda;\mathbb{R})$ represented by the $\Lambda$-equivariant $d$-form $\eta$. Denote by $M_\Lambda$ the expanding matrix associated to $\Lambda$. Then $\Phi^*_\Lambda[\eta]$ is represented by the form $M_\Lambda^*\eta$, where by $M^*_\Lambda$ we denote the pull-back by the linear map $M^*_\Lambda$.

Let $E_i$ be the generalized eigenspaces for the action of $\Phi_{\Lambda}^*$ on $H^d(\Omega_{\Lambda};\mathbb{R})$ induced by the map by $\Phi_\Lambda$ corresponding to the eigenvalue $\nu_i$. The subspaces $E_i$ are decomposed as
$$E_i = \bigoplus_{j=1}^{\kappa(i)} E_{i,j},$$
where $\kappa(i)$ is the size of the largest Jordan block associated with $\nu_i$. For each $i$, we choose a basis of classes $\{[\eta_{i,j,k}]\}$ with the property that $\langle [\eta_{i,j,1}],[\eta_{i,j,2}],\dots, [\eta_{i,j,s(i,j)}]\rangle = E_{i,j}$ and
\begin{equation}
  \label{eqn:basisAct}
  \Phi_\Lambda^* [\eta_{i,j,k}] = \left\{\begin{array}{ll}
\nu_i [\eta_{i,j,k}] + [\eta_{i,j-1,k}]   &\mbox{ for $j>1$,} \\
\nu_i [\eta_{i,j,k}]  &\mbox{ for $j=1$.}
  \end{array}\right.
\end{equation}

\begin{definition}
The \textbf{rapidly expanding subspace} $E^+_\Lambda(\Omega_\Lambda) \subset H^d(\Omega_\Lambda;\mathbb{R})$ is the direct sum of all generalized eigenspaces $E_i$ of $\Phi_\Lambda^*$ such that the corresponding eigenvalues $\nu_i$ of $\Phi_{\Lambda}^*$ satisfy
\begin{equation}
\label{eqn:RES}
\frac{\log |\nu_i|}{\log \nu_1} \geq 1 - \frac{\log|\lambda_d|}{\log \nu_1}.
\end{equation}
\end{definition}
Let us briefly comment on what (\ref{eqn:RES}) is meant to capture. We will be approximating the integrated density of states through ergodic integrals. This means we will be integrating certain functions on larger and larger sets. As such, the boundary of these sets will have certain contribution to the ergodic integral. The inequality (\ref{eqn:RES}) captures classes which are represented by forms whose growth in these types of integrals is greater than the contribution to the integral given by the boundary. See (\ref{eqn:ergExpansion}) and (\ref{eqn:intBound})for more details.

We set $I^+_\Lambda = I^{+,>}_\Lambda \cup I^{+,=}_\Lambda$ be the index set of classes $[\eta_{i,j,k}]$ which form a generalized eigenbasis for $E^+_\Lambda$, where the indices in $I^{+,>}_\Lambda$ contains vectors corresponding to a strict inequality in (\ref{eqn:RES}) and the indices in $I^{+,=}_\Lambda$ correspond to vectors associated to eigenvalues which give an equality in (\ref{eqn:RES}). Note that $I^{+,=}_\Lambda$ can be empty but $I^{+,>}_\Lambda$ always has at least one element.

Let $\eta_{i,j,k}\in\Delta_\Lambda^d$ be a representative of the class $[\eta_{i,j,k}]$ in the eigenbasis in (\ref{eqn:basisAct}). By \cite[\S 4]{ST:SA}, there exist forms $\zeta_{i,j,k}\in\Delta_\Lambda^{d-1}$ such that
\begin{equation}
\label{eqn:ACtionExact}
M_\Lambda^*\eta_{i,j,k} = \nu_i \eta_{i,j,k} + \eta_{i,j-1,k} + d\zeta_{i,j,k}.
\end{equation}
For any $\eta\in\Delta^d_\Lambda$, we denote by $\alpha_{i,j,k}(\eta)$ the component of the class $[\eta]$ in the subspace spanned by $[\eta_{i,j,k}]$. In other words, $\eta = \sum_{i,j,k}\alpha_{i,j,k}(\eta)\eta_{i,j,k} + d\omega_\eta$ for some $\omega_\eta\in\Delta_\Lambda^{d-1}$. For any $f\in\Delta_\Lambda^0$ the component $\alpha_{i,j,k}(f)$ is defined by duality: $\alpha_{i,j,k}(f) := \alpha_{i,j,k}(f(\star 1))$.

It was showed in \cite{ST:SA} that elements in the dual space $\mathfrak{C}_{\Lambda}^+ = (E_\Lambda^+)' $ to the rapidly expanding subspace $E_\Lambda^+$ are represented by $\mathbb{R}^d$-invariant $\Lambda$-equivariant currents. It admits a decomposition into eigenvectors of the induced action by $\Phi_\Lambda$:
\begin{equation}
  \label{eqn:homology}
  \mathfrak{C}_\Lambda^+ = \bigoplus_{(i,j,k)\in I^{+}_\Lambda}\, \mathrm{span}\, \mathfrak{C}_{i,j,k},
\end{equation}
where the $\mathfrak{C}_{i,j,k}\in (\Delta_\Lambda^d)'$. We can identify $\mathfrak{C}_\Lambda^+$ with a subspace of the homology of $\Omega_\Lambda$. The currents $\mathfrak{C}_{i,j,k}\in\mathfrak{C}_\Lambda^+$ define distributions $\Gamma_{i,j,k}$ on $C^\infty(\mho_\Lambda)$ as follows. Recall that a function $g\in C^{\infty}(\mho_\Lambda)$ has a cohomology class $[g]\in H^d(\Omega_\Lambda;\mathbb{R})$. The distributions $\Gamma_{i,j,k}$ are defined as
$$\Gamma_{i,j,k}(g):= \mathfrak{C}_{i,j,k}([g])$$
for any $g\in C^{\infty}(\mho_\Lambda)$. We point out that we have that $\Gamma_{1,1,1} = \mathfrak{m}_\Lambda$, where $\mathfrak{m}_\Lambda$ is the canonical measure on $\mho_\Lambda$ coming from the $\mathbb{R}^d$-invariant measure on $\Omega_\Lambda$ (this will be explained in \S \ref{sec:TracesCycles}).
\section{Algebras and cohomology}
\label{sec:ops}
In this section we recall the setup from \cite{LS:algebras} which will allow us to work with operators of the right kind. Let
$$\mathcal{G}_\Lambda(g) = \{(p,\Lambda',q)\in \mathbb{R}^d\times \Omega_\Lambda\times \mathbb{R}^d: p,q\in \Lambda' \}.$$
\begin{definition}
\label{def:kernel}
A \textbf{kernel of finite range} is a function $k\in C(\mathcal{G}_\Lambda)$ such that:
\begin{enumerate}
\item $k$ is bounded;
\item $k$ has finite range: there exists an $R_k>0$ such that $k(p,\Lambda',q) = 0$ whenever $|p-q|\geq R_k$;
\item $k$ is $\mathbb{R}^d$-invariant: for any $t\in\mathbb{R}^d$ we have that $k(p+t,\Lambda'+t,q+t) = k(p,\Lambda',q)$.
\end{enumerate}
\end{definition}
The set of all kernels of finite range is denoted by $\mathcal{K}^{fin}_\Lambda$. Note that for any two $\Lambda_1,\Lambda_2\in\Omega_\Lambda$ we have that $\mathcal{K}^{fin}_{\Lambda_1} = \mathcal{K}^{fin}_{\Lambda_2}$. For any kernel $k\in\mathcal{K}_\Lambda^{fin}$ there is a family of representations $\{\pi_{\Lambda'}\}_{\Lambda'\in\Omega_\Lambda}$ in $\mathcal{B}(\ell^2(\Lambda'))$ defined, for $k\in \mathcal{K}^{fin}_\Lambda$ by
$$\langle K_{\Lambda'} \delta_p,\delta_q\rangle  := \langle (\pi_{\Lambda'} k) \delta_p,\delta_q\rangle = k(p,\Lambda',q)$$
for $p,q\in\Lambda'$.

The family $\{K_{\Lambda'}\}_{\Lambda'\in\Omega_\Lambda}$ is bounded in the product $\Pi_{\Lambda'\in\Omega_\Lambda}\mathcal{B}(\ell^2(\Lambda'))$. To make $\mathcal{K}^{fin}_\Lambda$ a $*$-algebra, the convolution product is defined as
$$(a\cdot b)(p,\Lambda',q) = \sum_{x\in \Lambda'} a(p,\Lambda',x)b(x,\Lambda',q)$$
and the involution by $k^*(p,\Lambda',q) = \bar{k}(q,\Lambda',p)$. As such, the map $\pi:\mathcal{K}^{fin}_\Lambda\rightarrow \Pi_{\Lambda'\in\Omega_\Lambda}\mathcal{B}(\ell^2(\Lambda'))$ is a faithful $*$-representation. The image of this map is denoted by $\mathcal{A}^{fin}_\Lambda$ and it is the algebra of \textbf{operators of finite range}. The completion of this space under the norm $\|A\| = \sup_{\Lambda'\in\Omega_\Lambda} \|A_{\Lambda'}\|$ is denoted by $\mathcal{A}_\Lambda$.
\begin{remark}
As shown in \cite{LS:algebras}, the algebra $\mathcal{A}_\Lambda$ is closely related (in fact, $*$-isomorphic) to the algebras considered in \cite{BHZ:Hulls, kellendonk:noncom}.
\end{remark}
\begin{definition}
The set of \textbf{$\Lambda$-equivariant kernels of finite range} are the kernels of finite range $k\in\mathcal{K}_\Lambda^{fin}$ for which there exists a $R'_k>0$ such that if for two $\Lambda_1,\Lambda_2\in\Omega_\Lambda$ we have that $B_{R'_k}(0)\cap \Lambda_1 = B_{R'_k}(0)\cap \Lambda_2$ then for any $p,q\in B_{R'_k}(0)\cap \Lambda_1$ we have that $k(p,\Lambda_1,q) = k(p,\Lambda_2,q)$.
\end{definition}
We denote the set of $\Lambda$-equivariant kernels by $\mathcal{K}_\Lambda^{tlc}$ which is a $*$-subalgebra of $\mathcal{K}_\Lambda^{fin}$. The image of $\Lambda$-equivariant kernels $\mathcal{A}_\Lambda^{tlc} = \pi \mathcal{K}_\Lambda^{tlc}$ is the $*$-subalgebra of \textbf{$\Lambda$-equivariant operators} of finite range.
\begin{remark}
  \label{rem:interesting}
The algebra $\mathcal{A}_\Lambda^{tlc}$ of $\Lambda$-equivariant operators of finite range includes many of the operators of interest coming from physics. In particular, many Laplacian operators $\triangle\in \mathcal{B}(\ell^2(\Lambda))$ come from $\Lambda$-equivariant operators, and so do Hamiltonians of the form $H = -\triangle + V$, where $V$ is a $\Lambda$-equivariant potential.
\end{remark}
\begin{definition}
Let $\mathcal{A}$ be a $*$-algebra. A \textbf{trace} on $\mathcal{A}$ is a continuous linear functional $\tau:\mathcal{A}\rightarrow \mathbb{C}$ satisfying $\tau(ab) = \tau(ba)$ for any two $a,b\in\mathcal{A}$. The set of all traces of $\mathcal{A}$ forms a $\mathbb{C}$-vector space which we will denote by $\mathrm{Tr}(\mathcal{A})$.
\end{definition}
Let $u:\mathbb{R}^d\rightarrow \mathbb{R}$ be a smooth bump function with compact support and integral one. We now define a family of maps $w_{\Lambda',u}:\mathcal{A}_\Lambda^{tlc} \rightarrow \Delta_\Lambda^0$ parametrized by $\Lambda'\in\Omega_\Lambda$. By duality, these are also maps to $\Delta_\Lambda^d$. For $A = \pi k \in\mathcal{A}_\Lambda^{tlc}$ and $A_{\Lambda'} = \pi_{\Lambda'} k \in\mathcal{B}(\ell^2(\Lambda'))$ we define the map  $w_{\Lambda',u}:A_{\Lambda'}\mapsto f_{A_{\Lambda'}}$ by
\begin{equation}
\label{eqn:opMap}
f_{A_{\Lambda'}}(t) = w_{\Lambda',u}(A)(t) := \sum_{p\in\Lambda'} A_{\Lambda'}(p,p)u(p - t),
\end{equation}
which is a smooth $\Lambda$-equivariant function. As such, it has a cohomology class.
\begin{definition}
\label{def:opClass}
Let $A\in\mathcal{A}_\Lambda^{tlc}$ be a $\Lambda$-equivariant operator of finite range. The \textbf{cohomology class $[A_{\Lambda'}]$ of the operator $A_{\Lambda'} = \pi_{\Lambda'} a\in \mathcal{B}(\ell^2(\Lambda'))$} is defined to be the cohomology class $[A_{\Lambda'}] = [f_{A_{\Lambda'}}(\star 1)] = [w_{\Lambda',u}(A)(\star 1)]\in H^d(\Omega_\Lambda;\mathbb{R}^d)$.
\end{definition}
\begin{remark}
Note that to construct a $\Lambda$-equivariant function $w_{u}(A)(\star 1)$ from $A\in \mathcal{A}_\Lambda^{tlc}$, and therefore to assign it a cohomology class, we used a smooth compactly supported function $u$. However, in the definition of the cohomology class of $A$ there is no reference to $u$. We will see in Proposition \ref{prop:traces} that for the purposes needed, this class is independent of which $u$ we take, as long as it is smooth, compactly supported, and has integral one.
\end{remark}
\begin{lemma}
  \label{lem:independence}
Let $A\in \mathcal{A}_\Lambda^{tlc}$ be a $\Lambda$-equivariant operator of finite range. For any two $\Lambda_1, \Lambda_2\in\Omega_\Lambda$ we have that $[A_{\Lambda_1}] = [A_{\Lambda_2}]$.
\end{lemma}
\begin{proof}
  Let $u$ be a radially symmetric smooth bump function supported in a very small ball (smaller than the inner radius of $\Lambda$) and of integral one. Consider the functions $f_i(t) = f_{A_{\Lambda_i}}(t)$ for $i = 1, 2$, where $A_{\Lambda_i} = \pi_{\Lambda_i}A$ and $A\in\mathcal{A}_\Lambda^{tlc}$. Then
$$f_i(t) = \sum_{p\in\Lambda_i}A_{\Lambda_i}(p,p)u(p-t).$$

  For $i = 1,2$, consider the functions $h_i:\mho_\Lambda\rightarrow \mathbb{R}$ defined as follows. For each $p\in\Lambda_i$ we can identify it to a point in $\mho_\Lambda$ by translating $\Lambda_i$ in such a way that $p$ is translated to $\bar{0}$. Call this map $m_i:\Lambda_i\rightarrow \mho_\Lambda$. The map $m_i$ is a bijection onto its image $m_i(\Lambda_i)$, which is dense in $\mho_\Lambda$. Let $h_i$ be the function $h_i:m_i(\Lambda_i) \rightarrow \mathbb{R}$ be defined by $h_i(\Lambda_0) = A(m^{-1}_i(\Lambda_0),m^{-1}_i(\Lambda_0))$ for each $\Lambda_0\in m_i(\Lambda_i)$. This function can be extended to the entire transversal $\mho_\Lambda$ as follows.

Recall that since $A_{\Lambda_i}  = \pi_{\Lambda_i} A$ come from a $\Lambda$-equivariant operators of finite range $A\in\mathcal{A}_\Lambda^{tlc}$ then there exists an $R_A$ such that if $\Lambda_a, \Lambda_b\in\Omega_\Lambda$ have the property that $\Lambda_a\cap B_{R_A}(0) = \Lambda_b\cap B_{R_A}(0)$ then $A_{\Lambda_a}(p,q) = A_{\Lambda_b}(p,q)$ for any $p,q\in B_{R_A}(0)\cap \Lambda_a$. For any $\Lambda_0 \in m_i(\Lambda_i)$, let $U_{R_A}(p)$ denote the $1/R_A$-neighborhood of $\Lambda_0$, that is,
$$U_{R_A}(\Lambda_0) = \{\Lambda'\in\mho_\Lambda: \Lambda' \cap B_{R_A}(0) = \Lambda_0 \cap B_{R_A}(0)\}.$$
As such, for any $\Lambda_a \in m_i(\Lambda_i)$ and any $\Lambda_b\in U_{R_A}(\Lambda_a)\cap m_i(\Lambda_i)$, we have that $h_i(\Lambda_a) = h_i(\Lambda_b)$. Since $m_i(\Lambda_i)\cap U_{R_A}(\Lambda_a)$ is dense in $U_{R_A}(\Lambda_a)$, we can extend $h_i$ to all of $U_{R_A}(\Lambda_a)$ since it is constant on the set $m_i(\Lambda_i)\cap U_{R_A}(\Lambda_a)$, and so we make $h_i$ constant on all of $U_{R_A}(\Lambda_a)$. Since $\mho_\Lambda$ is compact, there are finitely many open sets of the form $U_{R_A}(\Lambda_0)$ for which we need to do this, and thus we can extend $h_i$ to be defined on all of $\mho_\Lambda$. Moreover, it is transversally locally constant.

Comparing $h_1$ and $h_2$, we see that if for two $\Lambda_a, \Lambda_b\in\mho_\Lambda$ we have that if $\Lambda_a\in U_{R_A}(\Lambda_b)$, $h_1(\Lambda_a)$ is determined by the value of $A$ determined by the cluster $B_{R_A}(0)\cap \Lambda_a$. Since $h_2(\Lambda_a)$ is also determined by the same clusters and the same $\Lambda$-equivariant operator $A$, $h_2(\Lambda_a) = h_1(\Lambda_b) = h_1(\Lambda_a)$. Since this argument works for any clopen neighborhood of the form $U_{R_A}(\Lambda_0)$ and $\mho_\Lambda$ is compact, $h_1(\Lambda_0) = h_2(\Lambda_0)$ for any $\Lambda_0\in\mho_\Lambda$. So the functions $h_1$ and $h_2$ are the same transversally locally constant functions, and $f_1, f_2$ are obtained by smoothing $h_1,h_2$ along the leaf direction with $u$. As such, they belong to the same cohomology class.
\end{proof}
\begin{remark}
Given Lemma \ref{lem:independence} we will supress the subscript $\Lambda'$ from the operators $A_{\Lambda'}$ because the cohomology class of their traces will be independent of representative.
\end{remark}
By Lemma \ref{lem:independence}, the map $w_{\Lambda',u}$ does not depend on $\Lambda'$ and thus we get a map $w_{u}:\mathcal{A}_\Lambda^{tlc}\rightarrow H^d(\Omega_\Lambda;\mathbb{R})$.

For any $A\in\mathcal{A}_\Lambda^{fin}$ and any bounded set $B\subset \mathbb{R}^d$, we denote by $A_{\Lambda'}|_B$ the restriction of $A_{\Lambda'}$ to the finite dimensional subspace $\ell^2(\Lambda'\cap B)$ of $\ell^2(\Lambda')$.
\section{Traces and asymptotic cycles}
\label{sec:TracesCycles}
We now recall the relevant ergodic theoretic results from \cite{ST:SA}. Let $\Lambda$ be an RFT Delone set and recall the definition of the rapidly expanding subspace $E^+_\Lambda\subset H^d(\Omega_\Lambda;\mathbb{R}^d)$. For a RFT Delone set there exists an expansive matrix $A\in GL^+(d,\mathbb{R}):= \exp(\mathfrak{gl}(d,\mathbb{R}))$ satisfying the conjugacy equation (\ref{eqn:conjugacy}). For any set $B_0\subset \mathbb{R}^d$ with non-empty interior, containing the origin and with a regular boundary (a good Lipschitz domain as defined in \cite[\S 5]{ST:SA}), we define a one-parameter family of sets for $T>1$:
\begin{equation}
\label{eqn:sets}
  B_T = exp\left(d\frac{a\log T}{\log|A|}\right)B_0.
\end{equation}
Here $a\in\mathfrak{gl}(d,\mathbb{R})$ satisfies $A = exp(a)$. These sets have the property that $\mathrm{Vol}(B_T) = \mathrm{Vol}(B_0)\cdot T^d$.

Using the basis in (\ref{eqn:basisAct}), for a good Lipschitz domain $B_0$ containing the origin and $\Lambda_0\in\Omega_\Lambda$, for any $f\in C^\infty_{tlc}(\Omega_\Lambda)$ we can write its ergodic integral as
\begin{equation}
  \label{eqn:ergExpansion}
\int_{B_T} f\circ \varphi_t(\Lambda_0)\, dt =  \sum_{(i,j,k)\in I^+_\Lambda}\alpha_{i,j,k}(f) \int_{B_T} \eta_{i,j,k} +  \mathcal{O}(|\partial B_T|),
\end{equation}
where the $\eta_{i,j,k}$ are the representatives of the basis in (\ref{eqn:basisAct}), and $\alpha_{i,j,k}(f)$ only depend on the cohomology class of $f$. By \cite[Lemma 5]{ST:SA} we have that $\mathrm{Vol}(\partial B_T)$ is proportional to $T^{d\left(1-\frac{\log|\lambda_d|}{\log \nu_1}\right)}$, thus $\mathcal{O}(|\partial B_T|)$ in the expansion above includes the contributions to the ergodic integral of growth comparable to the volume of the boundary $\partial B_T$. Moreover, by \cite[Proposition 6]{ST:SA},
\begin{equation}
  \label{eqn:intBound}
\left| \alpha_{i,j,k}(f) \int_{B_T} \eta_{i,j,k} \right| \leq C_{\Lambda, B_0}|\alpha_{i,j,k}(f)|L(i,j,T) T^{d\frac{\log |\nu_i|}{\log \nu_1}},
\end{equation}
where $L(i,j,T)$ is a non-negative power of $\log T$ (see (\ref{eqn:currents1}) and (\ref{eqn:currents2}) below), for all $T>0$. The main result of \cite{ST:SA} states that there exist $|I^+_\Lambda|$ closed, $\mathbb{R}^d$-invariant, $\Lambda$-equivariant currents $\{\mathfrak{C}_{i,j,k}\}_{(i,j,k)\in I^+_\Lambda}$ which control the growth of ergodic integrals (\ref{eqn:ergExpansion}).

We now recall the construction of the asymptotic cycles from \cite[\S 5.3]{ST:SA}. For $T>3$ and an index $(i,j,k)\in I_\Lambda^+$, define the $\Lambda$-equivariant currents $\mathfrak{C}^{B_0,T}_{i,j,k}$as
\begin{equation}
\label{eqn:functl1}
\begin{split}
\mathfrak{C}_{i,j,k}^{B_0,T} : \eta\mapsto \mathfrak{C}_{i,j,k}^{B_0,T}(\eta) &=  \int_{B_T} \eta -
\sum_{\substack{(i',j',k')\leq (i,j,k) \\ k'\neq k}}\alpha_{i',j',k'}(\eta) \int_{B_T} \eta_{i',j',k'}\\
&= \sum_{\substack{(i',j',k')\geq (i,j,k) \\ k'\neq k}}\alpha_{i',j',k'}(\eta) \int_{B_T} \eta_{i',j',k'} + \mathcal{O}(|\partial B_T|)
\end{split}
\end{equation}
for any $\eta\in\Delta_\Lambda^d$. Let $s_i = d\frac{\log |\nu_{i}|}{\log|\nu_1|}$. We can now average to define the currents. For an index $(i,j,k)\in I_\Lambda^{+,>}$,
\begin{equation}
\label{eqn:currents1}
  \mathfrak{C}_{i,j,k}([\eta]) = \limsup_{T\rightarrow \infty} \frac{1}{(\log T)^{j-1} T^{s_i}}\mathfrak{C}_{i,j,k}^{B_0,T}(\eta),
\end{equation}
which, by (\ref{eqn:intBound}), exists. In this case $L(i,j,T) = (\log T)^{j-1}$. For an index $(i,j,k)\in I_\Lambda^{+,=}$,
\begin{equation}
\label{eqn:currents2}
  \mathfrak{C}_{i,j,k}([\eta]) = \limsup_{T\rightarrow \infty} \frac{1}{(\log T)^{j} T^{s_i}}\mathfrak{C}_{i,j,k}^{B_0,T}(\eta).
\end{equation}
In this case $L(i,j,T) = (\log T)^{j}$.

In the notation above, we emphasize that the functionals $\mathfrak{C}_{i,j,k}$ only depend on the cohomology class $[f]$ of $f$, which makes them cycles. These functionals yield asymptotic cycles $\mathfrak{C}_{i,j,k}$ in the sense that they are defined by an averaging procedure along orbits of the $\mathbb{R}^d$ action. In fact, $\mathfrak{C}_{1,1,1}$ is the Schwartzman-Ruelle-Sullivan asymptotic cycle corresponding to the leading eigenvalue $\nu_1 = |A|$. Close examination will convince the reader that they also satisfy $\mathfrak{C}_{i,j,k}([f]) = 0 $ if and only if $ \alpha_{i,j,k}(f)=0$. In fact, $\mathfrak{C}_{i,j,k}$ is a non-zero multiple of $\alpha_{i,j,k}$. By scaling the forms $\eta_{i,j,k}$ appropriately, that is, by scaling $\eta_{i,j,k}$ such that
\begin{equation}
  \label{eqn:scaleForm}
  \limsup_{T\rightarrow \infty}\frac{1}{L(i,j,k)T^{s_i}}\int_{B_T}\eta_{i,j,k} = 1,
\end{equation}
we can assume that indeed $\mathfrak{C}_{i,j,k} = \alpha_{i,j,k}$.

For $(i,j,k)\in I^+_\Lambda$, define the map $\tau_{i,j,k}:\mathcal{A}_\Lambda^{tlc}\rightarrow \mathbb{R}$ by
\begin{equation}
\label{eqn:traces}
\tau_{i,j,k}: A\mapsto \mathfrak{C}_{i,j,k}([A]) = \mathfrak{C}_{i,j,k}([w_{\Lambda,u}(A)(\star 1)]).
\end{equation}
\begin{proposition}
\label{prop:traces}
The maps $\tau_{i,j,k}$ defined in (\ref{eqn:traces}) are traces on $\mathcal{A}_\Lambda^{tlc}$.
\end{proposition}
\begin{proof}
  Let $u:\mathbb{R}^d\rightarrow \mathbb{R}$ be a smooth function with compact support and integral one. We now turn to applying the currents $\mathfrak{C}_{i,j,k}$ to the forms $w_{\Lambda, u}(A)(\star 1)$ obtained through the map (\ref{eqn:opMap}) for operators $A\in\mathcal{A}_\Lambda^{tlc}$. It suffices to show that $\tau_{i,j,k}$ satisfies $\tau_{i,j,k}(AB) = \tau_{i,j,k}(BA)$ for $A = \{A_{\Lambda'}\}_{\Lambda'\in\Omega_\Lambda}$ and $B = \{B_{\Lambda'}\}_{\Lambda'\in\Omega_\Lambda}$. Let $A,B\in\mathcal{A}_\Lambda^{tlc}$.

We first compute the images of $AB$ and $BA$, respectively, under the map $w_{\Lambda, u}$. Let $a,b \in \mathcal{K}_\Lambda^{tlc}$ satisfy $\pi_\Lambda a = A$ and $\pi_\Lambda b = B$. Recalling the convolution product and the $*$-involution for kernels of finite range, and (\ref{eqn:opMap}),
\begin{equation}
\label{eqn:first}
f_{AB}(t) = w_{\Lambda, u}(AB)(t)= \sum_{p\in\Lambda} (AB)_\Lambda(p,p)u(p-t) = \sum_{p\in\Lambda} \sum_{x\in\Lambda}a(p,\Lambda, x)b(x,\Lambda, p)   u(p-t).
\end{equation}
Likewise:
\begin{equation}
\label{eqn:second}
f_{BA}(t) = w_{\Lambda, u}(BA)(t)= \sum_{p\in\Lambda} (BA)_\Lambda(p,p)u(p-t) = \sum_{p\in\Lambda} \sum_{x\in\Lambda}b(p,\Lambda, x) a(x,\Lambda, p)  u(p-t).
\end{equation}
The difference is
$$f_{AB}(t) - f_{BA}(t) = \sum_{p\in \Lambda}\sum_{q\in\Lambda}\left( a(p,\Lambda,q)b(q,\Lambda,p) - b(p,\Lambda,q)a(q,\Lambda,p)   \right)u(p-t)  = f_{AB - BA}(t) .$$
Let
$$D_T = \int_{B_T} f_{AB - BA}(t)\, dt$$
and $R_*$ be the maximum of the ranges of $a\in \mathcal{K}_\Lambda^{tlc}$ and $b\in \mathcal{K}_\Lambda^{tlc}$. By definition, for any $p,q\in\Lambda$ with $|p-q| \geq R_*$, we have that $a(p,\Lambda,q)= b(p,\Lambda,q) = 0$. Denote by $r_u > 0$ a number such that the support of $u$ is contained in the ball of radius $r_u$ around the origin. Denote by $B^{A,B,r_u}_T\subset B_T$ the subset
$$B_T^{A,B,r_u} = \{x\in B_T:\mathrm{dist}(x,\partial B_T)> 2(r_u +R_*)  \},$$
which is not empty for all large enough $T>0$. For $r>0$ and a subset $E\subset\mathbb{R}^d$ let
$$\partial_r E = \{x\in \mathbb{R}^d:\mathrm{dist}(x,\partial E)\leq r  \}$$
be the $r$-neighborhood of $\partial E$.

Suppose $p_1\in\Lambda$ is such that $p_1 + B_{r_u}(0) \subset B_T^{A,B,r_u}$, and denote by $q_1(p_1),\dots, q_{N(p_1)}(p_1)$, the other points with which $p_1$ interacts. In other words, $a(p_1,\Lambda,q_i)b(q_i,\Lambda,p_1)\neq 0$ or $b(p_1,\Lambda,q_i)a(q_i,\Lambda,p_1)\neq 0$ for all $i$. Since the integral of $u$ is 1, then $p_1$ contributes
$$\sum_{i=1}^{N(p_1)} a(p_1,\Lambda,q_i)b(q_i,\Lambda,p_1) - b(p_1,\Lambda,q_i)a(q_i,\Lambda,p_1) $$
in the sum under the integral $D_T$.

Let $p_2$ be such that $a(p_1,\Lambda,p_2)b(p_2,\Lambda,p_1)\neq 0$ or $b(p_1,\Lambda,p_2)a(p_2,\Lambda,p_1)\neq 0$ and $|p_1-p_2|\leq R_*$. Let $q_1(p_2),\dots, q_{N(p_2)}(p_2)$ be the points with which $p_2$ interacts: $a(p_1,\Lambda,q_j)b(q_j,\Lambda,p_1)\neq 0$ or $b(p_1,\Lambda,q_j)a(q_j,\Lambda,p_1)\neq 0$ for all $i$. Then $p_2 = q_{i'}(p_1)$ for some $i'$ and $p_1 = q_{j'}(p_2)$ for some $j'$. Moreover, the contrubution of $p_2$ in the sum of $D_T$ is
$$\sum_{j=1}^{N(p_2)} a(p_2,\Lambda,q_j)b(q_j,\Lambda,p_2) - b(p_2,\Lambda,q_j)a(q_j,\Lambda,p_2) .$$
Now we can consider the combined contribution of $p_1$ and $p_2$ in the sum of $D_T$. It is
\begin{equation*}
  \label{eqn:joint}
  \begin{split}
 & a(p_1,\Lambda,q_{i'}(p_1))b(q_{i'}(p_1),\Lambda,p_1) - b(p_1,\Lambda,q_{i'}(p_1))a(q_{i'}(p_1),\Lambda,p_1) \\
  &\hspace{.4in}+ \sum_{\substack{i=1 \\ i\neq i'}}^{N(p_1)} a(p_1,\Lambda,q_i)b(q_i,\Lambda,p_1) - b(p_1,\Lambda,q_i)a(q_i,\Lambda,p_1) \\
   &\hspace{.4in}\hspace{.4in} + a(p_2,\Lambda,q_{j'}(p_2))b(q_{j'}(p_2),\Lambda,p_2) - b(p_2,\Lambda,q_{j'}(p_2))a(q_{j'}(p_2),\Lambda,p_2) \\
   &\hspace{.4in}\hspace{.4in}\hspace{.4in}+ \sum_{\substack{j=1 \\ j\neq j'}}^{N(p_2)} a(p_2,\Lambda,q_j)b(q_j,\Lambda,p_2) - b(p_2,\Lambda,q_j)a(q_j,\Lambda,p_2) \\
  =&\,\,\, a(p_1,\Lambda,p_2)b(p_2,\Lambda,p_1) - b(p_1,\Lambda,p_2)a(p_2,\Lambda,p_1) \\
  &\hspace{.4in}+ \sum_{\substack{i=1 \\ i\neq i'}}^{N(p_1)} a(p_1,\Lambda,q_i)b(q_i,\Lambda,p_1) - b(p_1,\Lambda,q_i)a(q_i,\Lambda,p_1) \\
   &\hspace{.4in}\hspace{.4in} + a(p_2,\Lambda,p_1)b(p_1,\Lambda,p_2) - b(p_2,\Lambda,p_1)a(p_1,\Lambda,p_2) \\
  &\hspace{.4in}\hspace{.8in}+ \sum_{\substack{j=1 \\ j\neq j'}}^{N(p_2)} a(p_2,\Lambda,q_j)b(q_j,\Lambda,p_2) - b(p_2,\Lambda,q_j)a(q_j,\Lambda,p_2) \\
  &= \sum_{\substack{i=1 \\ i\neq i'}}^{N(p_1)} a(p_1,\Lambda,q_i)b(q_i,\Lambda,p_1) - b(p_1,\Lambda,q_i)a(q_i,\Lambda,p_1)\\
  &\hspace{.6in}+\sum_{\substack{j=1 \\ j\neq j'}}^{N(p_2)} a(p_2,\Lambda,q_j)b(q_j,\Lambda,p_2) - b(p_2,\Lambda,q_j)a(q_j,\Lambda,p_2).
 \end{split}
\end{equation*}
In short, the interactions between $p_1$ and $p_2$ cancel out. Since this sort of cancellation happens for all pairs of points $p^*, q^*\in B_T^{A,B,r_u}$ with $a(p^*,\Lambda,q^*)b(q^*,\Lambda,p^*)\neq 0$ or $b(p^*,\Lambda,q^*)a(q^*,\Lambda,p^*)\neq 0$,
\begin{equation}
  \label{eqn:difference}
  \begin{split}
    \left| \int_{B_T} f_{AB - BA}(t)\, dt \right| &\leq  \left| \int_{B_T \backslash B_T^{A,B,r_u}} f_{AB - BA}(t)\, dt \right| \leq   \int_{\partial_{2(r_u + R_*)}B_T} |f_{AB - BA}(t)|\, dt  \\
    &\leq  \#( \partial_{2(r_u + R_*)}B_T \cap \Lambda) \max_{p\in\Lambda}|A(p,p)| \\
    &\leq 4(r_u + R_*)D_{\Lambda}\|A\| \mathrm{Vol}(B_0)T^{d\left(1-\frac{\log\lambda_d}{\log\nu_1}\right)} = \mathcal{O}(|\partial B_T|)
  \end{split}
\end{equation}
for all $T$ large enough, where we have used \cite[Lemma 5]{ST:SA} in the last inequality and $D_\Lambda$ only depends on the Delone set $\Lambda$. Comparing (\ref{eqn:difference}) with its expansion through (\ref{eqn:ergExpansion}), we have that $\alpha_{i,j,k}(f_{AB-BA}) = 0$ for all $(i,j,k)\in I^+_\Lambda$ and thus  $\mathfrak{C}_{i,j,k}([AB-BA]) = 0$ for all $(i,j,k)\in I_\Lambda^+$. As such, $\tau_{i,j,k}(AB) = \mathfrak{C}_{i,j,k}([AB])=\mathfrak{C}_{i,j,k}([BA]) = \tau_{i,j,k}(BA)$. The proposition follows since this works for any bump function $u$.
\end{proof}

Let $\tau_{i,j,k} = \mathfrak{C}_{i,j,k}\circ w:\mathcal{A}_\Lambda^{tlc}\rightarrow \mathbb{C}$ be the traces above.
\begin{corollary}
  \label{cor:traces}
  The space
$$    \mathrm{Tr}^+_\Lambda(\mathcal{A}_\Lambda^{tlc}) := \bigoplus_{(i,j,k)\in I^{+}_\Lambda}\, \mathrm{span}\, \tau_{i,j,k} $$
is a subspace of dimension $\mathrm{dim}\,E^+_\Lambda$ of the space of traces $\mathrm{Tr}(\mathcal{A}_\Lambda^{tlc})$.
\end{corollary}
\begin{proof}[Proof of Theorem \ref{thm:main}]
Let $B_0$ be a good Lipschitz domain, $A\in \mathcal{A}_\Lambda^{tlc}$, and $u$ a smooth bump function of compact support and integral one. Using the map in (\ref{eqn:opMap}), let $f_A(t) = w_{\Lambda,u}(A)(t)$ be a $\Lambda$-equivariant function and denote
$$I_T = \int_{B_T} f_A(t)\, dt.$$
We want to bound the quantity
\begin{equation*}
  \begin{split}
    \left|  \sum_{q\in \Lambda\cap B_T} A(q,q)  -  I_T \right| &= \left|  \sum_{q\in \Lambda\cap B_T} A(q,q) - \int_{B_T} f_A(t)\, dt \right| \\
    &= \left|  \sum_{p\in \Lambda\cap B_T} A(p,p) - \int_{B_T} \sum_{p\in \Lambda} A(p,p)u(p-t) \, dt \right| .
  \end{split}
\end{equation*}
Let $r_u > 0$ be such that $\mathrm{supp}(u)\subset B_{r_u}(0)$ and let $T$ be large enough so that $B_{r_u}(0)\subset B_T$. Suppose $p\in\Lambda$ is such that $p + B_{r_u}(0) \subset B_T$. Since the integral of $u$ is 1, then $p$ contributes $A(p,p)$ in the sum under the integral $I_T$, and so it cancels with the same quantity in $\sum_{p\in\Lambda} A(p,p)$. This happens to all $p\in \Lambda\cap B_T$ with the exception of those $q\in\Lambda\cap B_T$ with distance to $\partial B_T$ less than or equal to $2r_u$. Thus
\begin{equation}
  \label{eqn:traces3}
  \begin{split}
    \left|  \sum_{q\in \Lambda\cap B_T} A(q,q)  -  I_T \right| &= \left|  \sum_{p\in \Lambda\cap B_T} A(p,p) - \int_{B_T} \sum_{p\in \Lambda} A(p,p)u(p-t) \, dt \right| \\
    &\leq \int_{\partial_{3r_u} B_T} \left| \sum_{p\in\partial_{3r_u} B_T\cap \Lambda} A(p,p)u(p-t)\right| \, dt  \\
    &\leq \#(\partial_{3r_u}B_T\cap \Lambda) \|A\| \\
    &\leq 4r_u D_\Lambda \|A\|\mathrm{Vol}(B_0) T^{d\left(1-\frac{\log\lambda_d}{\log\nu_1}\right)} = \mathcal{O}(|\partial B_T|).
  \end{split}
\end{equation}
Therefore, up to terms of order $|\partial B_T|$, $\sum_{q\in \Lambda\cap B_T} A(q,q)$ and $\int_{B_T}f_A(t)\, dt$ agree. For every index $(i,j,k)\in I^+_\Lambda$ define the function
\begin{equation}
  \label{eqn:PsiFuns}
\Psi_{i,j,k}^{B_0}(T) :=\frac{1}{L(i,j,T)T^{\frac{\log |\nu_i|}{\log\nu_1}}} \int_{B_T}\eta_{i,j,k}: \mathbb{R}^+\rightarrow \mathbb{R}.
\end{equation}
By (\ref{eqn:scaleForm}), these functions satisfy $\limsup_{T\rightarrow\infty}\Psi_{i,j,k}^{B_0}(T) = 1$. Using (\ref{eqn:functl1}),
\begin{equation}
  \label{eqn:main1}
  \begin{split}
    \mathfrak{C}_{i,j,k}^{B_0,T}(f_A(\star 1))
    &= \int_{B_T} f_A(t)\, dt - \sum_{\substack{(i',j',k')\leq (i,j,k) \\ k'\neq k}}\alpha_{i',j',k'}([f_A]) \int_{B_T} \eta_{i',j',k'} \\
    &= \int_{B_T} f_A(t)\, dt - \sum_{\substack{(i',j',k')\leq (i,j,k) \\ k'\neq k}}\mathfrak{C}_{i',j',k'}([f_A]) \Psi_{i',j',k'}^{B_0}(T)L(i',j',T)T^{\frac{\log |\nu_{i'}|}{\log\nu_1}} \\
    &= \int_{B_T} f_A(t)\, dt - \sum_{\substack{(i',j',k')\leq (i,j,k) \\ k'\neq k}}\tau_{i',j',k'}(A) \Psi_{i',j',k'}^{B_0}(T)L(i',j',T)T^{\frac{\log |\nu_{i'}|}{\log\nu_1}} \\
    &= \mathrm{tr}(A|_{B_T}) - \sum_{\substack{(i',j',k')\leq (i,j,k) \\ k'\neq k}}\tau_{i',j',k'}(A) \Psi_{i',j',k'}^{B_0}(T)L(i',j',T)T^{\frac{\log |\nu_{i'}|}{\log\nu_1}}  + \mathcal{O}(|\partial B_T|).
  \end{split}
\end{equation}
As such, (\ref{eqn:main}) is obtained through (\ref{eqn:currents1}), (\ref{eqn:currents2}) and (\ref{eqn:main1}).
\end{proof}
Let $\mathcal{S}_\Lambda\subset \mathcal{A}_\Lambda^{tlc}$ be the subset of self-adjoint elements of $\mathcal{A}_\Lambda^{tlc}$. That is, for $A\in \mathcal{S}_\Lambda$, $A_{\Lambda'}\in \mathcal{B}(\ell^2(\Lambda'))$ is self-adjoint and $\Lambda$-equivariant for every $\Lambda'\in\Omega_\Lambda$. Recall that for any any self-adjoint operator $A\in\mathcal{B}(\ell^2(\Lambda))$ we can construct the C$^*$-algebra $C(A) = C(A,1)$ which is generated by $A$ and the identity $1$. That is, it is the completion of $P(A)$ in the operator norm of the set of all polynomials in $A$. The continuous functional calculus states that this algebra is $*$-isomorphic to $C(\sigma(A))$.
\begin{proposition}
\label{prop:injectiveMap}
Let $A\in\mathcal{S}_\Lambda$ be a family of self-adjoint operators and denote by $A_{\Lambda'}$ the associated self-adjoint operator in $\mathcal{B}(\ell^2(\Lambda'))$ for any $\Lambda'\in\Omega_\Lambda$. Then for any $\Lambda'\in\Omega_\Lambda$ there is an injective map $\Theta_{A_{\Lambda'}}:\mathrm{Tr}^+_\Lambda(\mathcal{A}_\Lambda^{tlc}) \rightarrow \mathrm{Tr}(C(A_{\Lambda'},1))$.
\end{proposition}
\begin{proof}
Let $\varphi(A)\in P(A)$ be a polynomial in $A$. Then $\varphi(A)$ is a $\Lambda$-equivariant operator of finite range and $\tau_{i,j,k}(\varphi(A))$ is well defined for any $\tau_{i,j,k}\in\mathrm{Tr}_\Lambda^+(\mathcal{A}_\Lambda^{tlc})$. Let $\varphi\in C([-\|A\|-2,\|A\|+2])$ be a continuous function and denote by $\{\varphi_n\}$ a Cauchy sequence of polynomials which converge to $\varphi$ uniformly in $C([-\|A\|-2,\|A\|+2])$ under the supremum norm. By the continuous functional calculus, $\{\tau_{i,j,k}(\varphi_n(A))\}_n$ is a Cauchy sequence, so $\tau_{i,j,k}(\varphi(A)) = \lim_{n\rightarrow \infty} \tau_{i,j,k}(\varphi_n(A))$ is the extension to $\mathrm{Tr}(C(A_{\Lambda'},1))$ of $\mathrm{Tr}_\Lambda^+(\mathcal{A}_\Lambda^{tlc})$.
\end{proof}
Given a self-adjoint operator $A\in \mathcal{S}_\Lambda$ let $J_A\subset \mathbb{R}$ be a closed interval of finite length which contains the spectrum $\sigma(A)$ of $A$ and let $\Lambda'\in\Omega_\Lambda$. For each index $(i,j,k)\in I_\Lambda^+$, by Proposition \ref{prop:injectiveMap}, there exists a unique regular countably additive Borel measure $\rho_{i,j,k}^{A}$ defined by
$$\rho_{i,j,k}^{A}(\varphi) = \Theta_{\Lambda'}(\tau_{i,j,k})(\varphi(A))$$
for any $\varphi\in C(J_A)$. In a slight abuse of notation, we will sometimes denote $\rho_{i,j,k}^{A}(\varphi) =\tau_{i,j,k}(\varphi(A))$ even though we implicitly use the map $\Theta_{\Lambda'}$ to extend the traces.
\begin{proof}[Proof of Theorem \ref{thm:shubin}]
  For a RFT Delone set $\Lambda'\in\Omega_\Lambda$ and a good Lipschitz domain $B_0$, let $A\in \mathcal{B}(\ell^2(\Lambda'))$ be defined by a self-adjoint $A\in \mathcal{S}_\Lambda$, and let $\varphi\in C(J_A)$ be a polynomial. We take the functions $\Psi_{i,j,k}^{B_0}$ to be the same ones as in (\ref{eqn:PsiFuns}). In \cite[Theorem 4.7]{LS:algebras}, it is shown that (see the end of the proof of Theorem 4.7)
  \begin{equation}
    \label{eqn:inOut}
    |\mathrm{tr}(\varphi(A|_{B_T}))  -  \mathrm{tr}(\varphi(A)|_{B_T})| \le C |\partial_{N\cdot R_{a}}B_{T}|,
\end{equation}
where $N$ is the degree of $\varphi$, and $R_{a}$ denotes the range of the kernel corresponding to $A$. We note that the term on the right hand side of (\ref{eqn:inOut}) is $\mathcal{O}(|\partial B_{T}|)$, since we have fixed $\varphi$. Thus, for $T>0$,
\begin{equation}
  \label{eqn:traces2}
  \begin{split}
    &\mathrm{tr}(\varphi(A|_{B_T})) - \sum_{\substack{(i',j',k')\leq (i,j,k) \\ k'\neq k}}\rho_{i',j',k'}^A(\varphi) \Psi_{i',j',k'}^{B_0}(T)L(i',j',T)T^{\frac{\log |\nu_{i'}|}{\log\nu_1}}   \\
    &\hspace{.25in}=\mathrm{tr}(\varphi(A)|_{B_T}) - \sum_{\substack{(i',j',k')\leq (i,j,k) \\ k'\neq k}}\tau_{i',j',k'}(\varphi(A)) \Psi_{i',j',k'}^{B_0}(T)L(i',j',T)T^{\frac{\log |\nu_{i'}|}{\log\nu_1}}  +\mathcal{O}(|\partial B_T|) \\
    &\hspace{.25in}=\mathrm{tr}(\varphi(A)|_{B_T}) - \sum_{\substack{(i',j',k')\leq (i,j,k) \\ k'\neq k}}\mathfrak{C}_{i',j',k'}([f_{\varphi(A)}]) \Psi_{i',j',k'}^{B_0}(T)L(i',j',T)T^{\frac{\log |\nu_{i'}|}{\log\nu_1}} +\mathcal{O}(|\partial B_T|) \\
    &\hspace{.25in}=\mathrm{tr}(\varphi(A)|_{B_T}) - \sum_{\substack{(i',j',k')\leq (i,j,k) \\ k'\neq k}}\alpha_{i',j',k'}([f_{\varphi(A)}]) \Psi_{i',j',k'}^{B_0}(T)L(i',j',T)T^{\frac{\log |\nu_{i'}|}{\log\nu_1}} +\mathcal{O}(|\partial B_T|) \\
    &\hspace{.25in}= \mathfrak{C}_{i,j,k}^{B_0,T}([f_{\varphi(A)}]) + \mathcal{O}(|\partial B_T|).
  \end{split}
\end{equation}
Therefore,
\begin{equation}
  \label{eqn:traces5}
  \begin{split}
    \limsup_{T\rightarrow\infty}&\frac{1}{L(i,j,T)T^{d\frac{\log|\nu_i|}{\log \nu_1}}}\left(  \mathrm{tr}(\varphi(A|_{B_T})) - \sum_{\substack{(i',j',k')\leq (i,j,k) \\ k'\neq k}}\rho_{i',j',k'}^A(\varphi) \Psi_{i',j',k'}^{B_0}(T) L(i',j',T)T^{\frac{\log |\nu_{i'}|}{\log\nu_1}}\right) \\
 &= \limsup_{T\rightarrow\infty}\frac{  \mathfrak{C}_{i,j,k}^{B_0,T}([f_{\varphi(A)}(\star 1)])}{L(i,j,T)T^{d\frac{\log|\nu_i|}{\log \nu_1}}}\\
 &=  \tau_{i,j,k}(\varphi(A)) = \rho_{i,j,k}^A(\varphi),
  \end{split}
\end{equation}
which proves Theorem \ref{thm:shubin}.
\end{proof}
\begin{remark}
\label{remark:allcontinuous}
Suppose now we want to extend the functionals in Theorem \ref{thm:shubin} to all continuous functions given that they are defined for polynomials. Let $h\in C(-\|A\|-2,\|A\| + 2)$ and suppose $\varphi_k\rightarrow h$ is a Cauchy sequence of polynomials. Since the traces are given by (\ref{eqn:traces5}), we can look at
$$\mathfrak{C}_{i,j,k}^{B_0,T}([f_{\varphi_m(A)}]) - \mathfrak{C}_{i,j,k}^{B_0,T}([f_{\varphi_{m+n}(A)}]).$$
Thus, according to (\ref{eqn:functl1}) and (\ref{eqn:main1}), unless the projections satisfy $\alpha_{i',j',k'}([f_{\varphi_m(A)}]) = \alpha_{i',j',k'}([f_{\varphi_{m+n}(A)}])$ for all the right indices, the difference will grow like $\mathrm{Vol}(B_T) \sim T^d$, so averaging by the quantities $L(i,j,T)T^{d\frac{\log|\nu_i|}{\log\nu_1}}$ as in (\ref{eqn:traces5}) will not prevent it from being unbounded.
\end{remark}
\section{An example}
\label{subsec:examples}
As we mentioned in the introduction, the Penrose tiling gives examples of $RFT$ Delone sets with rapidly expanding subspaces of dimension greater than one. Here we will work out an example in one-dimension in full detail.

A one-dimensional substitution is called \textbf{proper} if every substituted letter begins with the same letter and if every substituted letter ends with the same letter. Here is an example on three symbols:
\begin{equation}
  \label{eqn:example}
  A\mapsto ABA,\;\;\;\; B \mapsto ACA,\;\;\;\; C\mapsto ABBCBBCBBCBBA.
\end{equation}
Beyond the combinatorial model given by (\ref{eqn:example}), we can build a geometric model as follows. Pick two letters in the alphabet, $L^-, L^+\in\{A,B,C\}$, and look at the infinite words given by applying the substitution infinitely many times to the seed $L^-.L^+$, expanding the words coming from $L^-$ to the left and to the right with the words coming from $L^+$. Here's an example for $(L^-,L^+) = (A,B)$:
$$A.B\mapsto ABA.ACA \mapsto ABAACAABA.ABAABBCBBCBBCBBAABA \mapsto \cdots .$$
Denote by $w(L^-,L^+)\in \{A,B,C\}^{\bar{\mathbb{Z}}}$, where $\bar{\mathbb{Z}} = \mathbb{Z}-\{0\}$, a resulting infinite word from this procedure, where the dot separating the positive part and negative part of the words is identified to $0$. Now we build a Delone set in $\mathbb{R}$ from $w = w(L^-,L^+)\in\{A,B,C\}^{\bar{\mathbb{Z}}}$. The points of $\Lambda$ are
$$\Lambda = \{0\}\cup\Lambda^+\cup \Lambda^- \hspace{.4in}\mbox{ where }\hspace{.4in}\Lambda^\pm :=  \bigcup_{i\in\mathbb{N}}\sum_{j=\pm1}^{\pm i} \pm\theta_{w_j},$$
and $(\theta_A,\theta_B,\theta_C) = (1,3,13)$, which are the lengths of the tiles in the geometric model. It is straight forward to check that this substitution has expansion 5, and that the tiles given by $\Lambda$ admit the substitution rule corresponding to (\ref{eqn:example}) with expansion 5. Since by \cite[Theorem 6.1]{sadun:book} $H^1(\Omega_\Lambda;\mathbb{R})$ is finite dimensional, $\Lambda$ is a RFT Delone set (here $M_\Lambda = 5$).

In one-dimension, there is an index in $I^+_\Lambda$ for every eigenvalue of norm greater or equal than one, and thus a trace $\tau_{i,j,k}$. By \cite[\S 6]{sadun:book}, the spectrum of the substitution matrix corresponding to a proper one-dimensional substitution is the same as the spectrum of the induced action of cohomology, so the eigenvalues of the induced action on $H^1(\Omega_\Lambda;\mathbb{R})$ are $\omega_1 = 5,\omega_2 = - 2,\mbox{ and } \omega_3 = 2$ while the associated eigenvectors are $v_1 = (1,3,13), v_2 = (1,-4,6), \mbox{ and }v_3 = (-1,0,2)$. Thus, in this example we have that the rapidly expanding subspace $E^+_\Lambda$ is in fact equal to $H^1(\Omega_\Lambda;\mathbb{R})$. Moreover, since the induced action is diagonalizable, we get three traces $\tau_{1,1,1}, \tau_{2,1,1}, \tau_{3,1,1,}$ associated to the eigenvalues $\omega_1,\omega_2, \omega_3$. So simplify notation we will write $\tau_{i} = \tau_{i,1,1}$.

Because $\mathbb{R}$ does not have much geometry, the Delone set $\Lambda$ can be readily seen to be in bijection with $\mathbb{Z}$. Indeed, if we label the points,
$$\Lambda_0 = 0,\hspace{.4in} \Lambda_i = \sum_{j=1}^{ i} \theta_{w_j}\hspace{.4in}\mbox{ for $i>0$ and }\hspace{.4in}  \Lambda_i = \sum_{j=-1}^{ i} -\theta_{w_j} $$
for $i<0$, then we can see the bijective correspondence. Thus, looking at operators in $\mathcal{B}(\ell^2(\Lambda))$ is the same as looking at operators in $\mathcal{B}(\ell^2(\mathbb{Z}))$.

Let us begin looking at the discrete Laplacian $\triangle_\Lambda\in\mathcal{B}(\ell^2(\Lambda))$. For $g\in \ell^2(\Lambda)$, it is given by
$$(\triangle_\Lambda g)_i = \frac{-g_{i-1}+2g_i-g_{i+1}}{2}.$$
For a smooth bump function $u$ with compact support and integral 1, its associated function given by (\ref{eqn:opMap}) is
$$f_{\triangle_\Lambda}(t) = \sum_{p\in\Lambda}\triangle_\Lambda(p,p)u(p-t) = \sum_{p\in\Lambda}u(p-t) = \sum_{p\in\Lambda}u* \delta_p.$$
The function associated to $\Delta_\Lambda$ plays a prominent role in mathematical physics: its autocorrelation function is used to compute the diffraction spectrum of the solid modeled by $\Lambda$ (see \cite[\S 6]{ST:SA}). By the arguments used in the proof of Theorem \ref{thm:main} and the expansion (\ref{eqn:ergExpansion}), we can see that
\begin{equation}
  \label{eqn:ShubEx1}
  \begin{split}
    \mathrm{tr}\left(\triangle_\Lambda|_{[-T,T]}\right) &= \int_{-T}^T f_{\triangle_\Lambda}(t)\, dt + O(1) = |\Lambda\cap [-T,T]|+ O(1) \\
    &= \tau_1(\triangle_\Lambda)T + \tau_2(\triangle_\Lambda)\Psi_2(T)T^{\frac{\log 2}{\log 5}} + \tau_3(\triangle_\Lambda)\Psi_3(T)T^{\frac{\log 2}{\log 5}} + O(1),
  \end{split}
\end{equation}
where $\Psi_2$ and $\Psi_3$ are bounded oscillating functions. Thus, Shubin's formula in this case reads
$$\frac{1}{2T}\mathrm{tr}(\triangle_\Lambda|_{[-T,T]})\longrightarrow \tau_{1}(\Delta_\Lambda) = \mathfrak{m}_\Lambda(\mho_\Lambda),$$
where $\mathfrak{m}_\Lambda(\mho_\Lambda)$ is the asymptotic frequence of $\Lambda$, in this case given by the measure of the canonical transversal $\mho_\Lambda$ by the unique transverse invariant measure $\mathfrak{m}_\Lambda$ to the $\mathbb{R}$ action. So the error of convergence in Shubin's formula is of order $T^{\frac{\log2}{\log 5}-1}$. Thus, although combinatorily all Laplacians on $\ell^2(\Lambda)$ for one-dimensional Delone sets are the same, Shubin's formula does pick up on the geometry of the Delone set and the structure of the pattern space $\Omega_\Lambda$.

For $L\in\{A,B,C\}$, let $N(L,T) = |\Lambda_L\cap [-T,T]|$ be the number of tiles of type $L$ of $\Lambda$ contained in the interval $[-T,T]$. The asymptotic frequency of the tyle of type $L$ is
$$\mathfrak{f}_L = \lim_{T\rightarrow\infty} \frac{N(L,T)}{2T},$$
and note that $\sum_L \mathfrak{f}_L = \mathfrak{m}_\Lambda(\mho_\Lambda) = \frac{5}{21}$ (this will be computed below).

We partition $\Lambda$ now into $\Lambda = \Lambda_A \cup \Lambda_B \cup \Lambda_C$ where $p\in \Lambda$ belongs to $\Lambda_L$ if $p$ is the left endpoint of a tile of type $L$. Let $V_A, V_B, V_C\in \mathcal{B}(\ell^2(\Lambda))$ be the (projection) operators defined for $L\in\{A,B,C\}$ by
$$V_L \delta_p = \left\{ \begin{array}{l} \delta_p \mbox{ if $p\in\Lambda_L$ ,}\\ 0\mbox{ otherwise}  \end{array}\right. \hspace{.7in}\mbox{ so that } \hspace{.7in}\langle V_L \delta_p,\delta_p\rangle  = \left\{ \begin{array}{l} 1 \mbox{ if $p\in\Lambda_L$ ,}\\ 0\mbox{ otherwise.}  \end{array}\right. $$
We can think of the operators $V_L$ as potentials localized on tiles of type $L$ in our material modeled by $\Lambda$. We now consider the three-parameter family of operators
$$H_{a,b,c} = \triangle_\Lambda + aV_A + bV_B + cV_C$$
with $a,b,c\in\mathbb{R}$. Let us look at the operator $H_0 := H_{-\mathfrak{f}_A, -\mathfrak{f}_B, -\mathfrak{f}_C}$.We have
$$f_{H_0}(t) = \sum_{p\in\Lambda}H_0(p,p)u(p-t) = \sum_{L\in\{A,B,C\}}\sum_{p\in \Lambda_L} H_0(p,p)u(p-t) = \sum_{L\in \{A,B,C\}}\sum_{p\in \Lambda_L}(\triangle_\Lambda(p,p)-\mathfrak{f}_L)u(p-t).$$
So we can compute the truncated traces as in (\ref{eqn:ShubEx1}):
\begin{equation}
\label{eqn:ShubEx2}
  \begin{split}
    \mathrm{tr}\left(H_0|_{[-T,T]}\right) &= \int_{-T}^T f_{H_0}(t)\, dt + O(1) = \sum_{L\in\{A,B,C\}}\int_{-T}^T \sum_{p\in\Lambda_L}\triangle_\Lambda(p,p)u(p-t)\, dt-2\mathfrak{f}_LT + O(1)\\
    &= \sum_{L\in\{A,B,C\}} |\Lambda_L\cap [-T,T]|  - 2\mathfrak{f}_L T + O(1) \\
    &= |\Lambda\cap [-T,T]|- 2\mathfrak{m}_\Lambda(\mho_\Lambda)T + O(1)
  \end{split}
\end{equation}
so that Shubin's formula reads
\begin{equation}
  \label{eqn:BadShubin}
\frac{1}{2T}\mathrm{tr}\left(H_0|_{[-T,T]}\right)  \rightarrow \tau_1(H_0) =  \mathfrak{m}_\Lambda(\mho_\Lambda) - \mathfrak{m}_\Lambda(\mho_\Lambda) = 0.
  \end{equation}
Denote by $[\chi_L]$ the cohomology class of the $\Lambda$-equivariant function $f_{V_L} = \sum_{p\in L}u*\delta_p$. The classes $[\chi_A], [\chi_B],[\chi_C]$ form a basis of $H^1(\Omega_\Lambda;\mathbb{R})$ which, in fact, is the standard basis. More precisely, the expressions for $v_1,v_2$ and $v_3$ above are done in terms of this (standard) basis. Having related these two bases for $H^1(\Omega_\Lambda;\mathbb{R})$ we can compute the change of basis maps and conclude that $[\chi_L]$ has non-zero component in both $v_2$ and $v_3$ for all $L$.

Note that $\mathfrak{C}_1([\chi_L]) = \mathfrak{f}_L$ and that $\mathfrak{C}_j(v_i) = 1$ if $i = j$ and $0$ otherwise, where $\mathfrak{C}_i$ are the 3 currents corresponding to the component of the projection of a class in $H^1(\Omega_\Lambda;\mathbb{R})$ to the corresponding eigenspace $v_i$ (see (\ref{eqn:currents2})-(\ref{eqn:traces})). Moreover, through our change of basis map we also conclude that
\begin{equation*}
  \begin{split}
    \mathfrak{C}_1([\chi_L])=  \left\{ \begin{array}{ll} \frac{2}{21} &\mbox{ if $L=A$,}\\ \frac{2}{21}&\mbox{ if $L=B$,} \\ \frac{1}{21}  &\mbox{ if $L=C$,}\end{array}\right.  &\hspace{1in} \mathfrak{C}_2([\chi_L])=  \left\{ \begin{array}{ll} \frac{1}{14} &\mbox{ if $L=A$,}\\ -\frac{5}{28}&\mbox{ if $L=B$,} \\ \frac{1}{28}  &\mbox{ if $L=C$,}\end{array}\right. \\
    \mbox{ and }  \hspace{.35in} \mathfrak{C}_3&([\chi_L]) = \left\{ \begin{array}{ll} -\frac{5}{6} &\mbox{ if $L=A$,}\\ -\frac{1}{12}&\mbox{ if $L=B$,} \\ \frac{1}{12}  &\mbox{ if $L=C$.}\end{array}\right.
    \end{split}
  \end{equation*}
Recall that we have $\tau_i(A) = \mathfrak{C}_i([f_A])$. Using the classes introduced in the previous paragraph, it is straight forward to work out that the class of $f_{H_0}$ decomposes as $[f_{H_0}] = \sum_L [\chi_L] - \mathfrak{f}_Lv_1$. Thus we can apply the traces
$$\tau_i(H_0) = \mathfrak{C}_i([f_{H_0}]) = \sum_L \mathfrak{C}_i([\chi_L]) - \mathfrak{f}_L\mathfrak{C}_i(v_1) ,$$
which shows again that $\tau_1(H_0) = 0$. However, it also shows that $\tau_2(H_0) = -\frac{1}{14} $ and $\tau_3(H_0) = -\frac{5}{6}$, thus justifying the non-trivial expansion as in (\ref{eqn:ShubEx1})
\begin{equation*}
  \begin{split}
    \mathrm{tr}\left(H_0|_{[-T,T]}\right) &= \tau_2(H_0)\Psi_2(T)T^{\frac{\log 2}{\log 5}} + \tau_3(H_0)\Psi_3(T)T^{\frac{\log 2}{\log 5}} + O(1) \\
    &= -\frac{\Psi_2(T)}{14}T^{\frac{\log 2}{\log 5}} - \frac{5\Psi_2(T)}{6}T^{\frac{\log 2}{\log 5}} + O(1),\\
  \end{split}
\end{equation*}
and so we see that even if Shubin's formula gives a zero asymptotic trace in (\ref{eqn:BadShubin}), the magnitude of the truncated traces for $H_0$ still grow, albeit at a slower rate.
\bibliographystyle{amsalpha}
\bibliography{biblio}

\end{document}